\documentclass[11pt]{article}
\usepackage[utf8]{inputenc}
\usepackage{graphicx,verbatim}
\usepackage{amssymb}
\usepackage{amsthm}
\usepackage{bbm}
\usepackage{amsmath}
\usepackage{mathtools}
\usepackage{textgreek}
\usepackage{amsfonts}
\usepackage{mathrsfs}
\usepackage[english]{babel}
\usepackage{lscape} 
\usepackage[vcentermath,enableskew]{youngtab}
\usepackage{tikz,tikz-cd}
\usepackage{tabu}
\usetikzlibrary{arrows,positioning,automata,shadows,fit,shapes}
\usepackage{hyperref}
\usepackage{array}   
\newcolumntype{R}{>{$}r<{$}} 
\newcolumntype{C}{>{$}c<{$}} 

\usepackage[english]{babel}
\usepackage{setspace}
\usepackage{microtype}
\usepackage[hang,flushmargin]{footmisc} 

\advance\textwidth by 1.2in \advance\oddsidemargin by -.6in \advance\evensidemargin by -.6in

\newtheorem{theorem}{Theorem}[section]
\newtheorem{definition}[theorem]{Definition}

\newtheorem{corollary}[theorem]{Corollary}
\newtheorem{lemma}[theorem]{Lemma}
\newtheorem{proposition}[theorem]{Proposition}
\newtheorem{remark}[theorem]{Remark}
\newtheorem*{theorem*}{Theorem}
\newtheorem*{definition*}{Definition}
\newtheorem*{lemma*}{Lemma}

\usepackage{chngcntr}
\counterwithin{table}{section}

\numberwithin{equation}{section}

\DeclareMathOperator{\End}{End}
\DeclareMathOperator{\im}{im}

\DeclareMathOperator{\ad}{ad}

\newcommand{\bbc}{\mathbb{C}}

\newcommand{\bbs}{\mathbb{S}}

\newcommand{\bbz}{\mathbb{Z}}

\newcommand{\Ca}{\mathcal{A}}

\newcommand{\Cc}{\mathcal{C}}
\newcommand{\Cd}{\mathcal{D}}

\newcommand{\Ci}{\mathcal{I}}

\newcommand{\Cm}{\mathcal{M}}

\newcommand{\Cp}{\mathcal{P}}

\newcommand{\Cu}{\mathcal{U}}

\newcommand{\Cw}{\mathcal{W}}
\newcommand{\Cx}{\mathcal{X}}

\newcommand{\Cz}{\mathcal{Z}}

\newcommand{\lie}[1]{\mathfrak{#1}}
\newcommand{\sqbinom}[2]{\genfrac{[}{]}{0pt}{}{#1}{#2}}
\newcommand{\deq}{\arrow[draw=none]{d}[sloped,auto=false]{=}}
\newcommand{\dop}{\arrow[draw=none]{d}[sloped,auto=false]{\oplus}}

\newtoggle{details}
\iftoggle{details}{%
  \newcommand{\details}[1]{
      \ \\
      {\color{blue}
        \textbf{Details:} #1
      }
      \ \\
  }
}{%
 \newcommand{\details}[1]{}
}

\title{On the Quantum Metaplectic Howe Duality}

\author{Matheus Brito, Marcelo De Martino}
\date{}

\begin{document}
\maketitle
\renewcommand{\thefootnote}{\fnsymbol{footnote}} 
\renewcommand{\thefootnote}{\arabic{footnote}} 
\renewcommand{\thefootnote}{\fnsymbol{footnote}} 
\renewcommand{\thefootnote}{\arabic{footnote}} 
\renewcommand{\thefootnote}{\fnsymbol{footnote}} 
\footnotetext{M. Brito - Universidade Federal do Parana - (mbrito@ufpr.br)}
\footnotetext{M. De Martino - University of Ghent (marcelo.goncalvesdemartino@ugent.be)}  
\renewcommand{\thefootnote}{\arabic{footnote}} 

\begin{abstract}
    We establish a quantum analogue of the classical metaplectic Howe duality involving the pair of Lie algebras $(\lie{sp}_{2n},\lie{sl}_2)$ in the case when $n=1$. Our
    results yield commuting representations of the pair of Drinfeld-Jimbo quantum groups $(\Cu_{q^2}(\lie{sl}_2),\Cu_{q}(\lie{sl}_2))$ realized in a suitable algebra of $q$-differential operators acting on the space of symplectic polynomial spinors. We obtain $q$-analogues for the symplectic Dirac operator, the Fischer decomposition, the expression for the symplectic polynomial monogenics and for the projection operators onto the monogenics. We also discuss $q$-analogues of generalized symmetries of the $q$-symplectic Dirac operator raising the homogeneous polynomial degree.
\end{abstract}

\section{Introduction}

The topic of Howe duality started with the influential works \cite{H89a,H89b} (see also \cite{H95}) in which Howe obtained a uniform description of the First Fundamental Theorem of Classical Invariants. His formulation was in terms of the Weyl algebra of partial differential operators with polynomial coefficients (and more generally a Weyl-Clifford algebra) acting on a suitable space of polynomial functions (respectively, polynomial-spinors). In his theory, the building blocks are pairs of reductive mutually-centralizing Lie (super)algebras inside the Weyl(-Clifford) algebra that, essentially, yield a multiplicity-free decomposition of the polynomial(-spinor) space.
Although Howe's treatment was done in a very general setting, considerable effort was put to describe these dualities in specific (families of) dual pairs (see, e.g., the books \cite{GW00,CW12} and the references therein for a detailed account on many of these dualities).
In the subsequent decades, several authors have investigated deformations of such theory to a quantum group context, both in a general way in terms of the Fundamental Theorems of Invariant Theory \cite{LZZ11,Z20}, but also in specialized, case-by-case ways, dealing with specific (families of) dual pairs at a time \cite{NUW96,IK01,FKZ19,BT23,Z03}. 

On the other hand, classical Howe dualities are specially useful in the context of Clifford analysis and their Hermitian extensions, where the algebras involved are described in terms of certain Dirac operators and double-covers of reductive groups. For an account on the orthogonal theory, see \cite{BDSES10} (and the references therein). For the symplectic theory see \cite{DBSS14,DBHS17}  and for its Hermitian extension, see \cite{EM24}. It is but natural to consider $q$-deformations of these theories and quantum analogues of Clifford analysis have been proposed, specially in the orthogonal case, in \cite{CS10,BSZ22}.

The present work concerns a initial step towards investigating $q$-deformations of symplectic Clifford analysis. We will achieve a complete $q$-analogue of the theory in the case where the rank of the symplectic algebra involved is one. Our formulation is in terms of $q$-differential operators as described in \cite{Ha90}. We shall use these $q$-Weyl algebras both in the polynomial side and also in the symplectic Clifford algebra side, and the end result is a commuting pair of Drinfeld-Jimbo quantum $\lie{sl}_2$'s acting on the space of symplectic polynomial-spinors in such a way that, when taking the limit $q\to 1$, we recover the results as in the classical case. Our setting is quite fortunate since, given the widely-known fact that the process of $q$-deformations of universal
enveloping algebras is not compatible with inclusion maps \cite[p.227]{NUW96}, it is not always the case that in the quantum group analogue of a given Howe dual pair, both members of the pair are of Drinfeld-Jimbo type (see \cite{NUW96,IK01}).

In more details, we now summarize our findings and explain the contents of each section. We start, in Section \ref{s:Prelim}, by recalling some fundamental ideas on quantum Calculus, in special some $q$-binomial identities that will be used throughout. We then introduce, in Section \ref{s:qWeyl}, the concept of $q$-Weyl algebras following the work of Hayashi in \cite{Ha90} and summarize some important commutation relations that will be crucial in realizing the dual pair. In Section \ref{s:Uqsl2} we recall the definitions of $\Cu_{q}(\lie{sl}_2)$ and some of its representation theory. Section \ref{s:qDuality} contains our main results. Namely, Theorem \ref{thm:spinact} on the realization, in terms of $q$-differential operators, of the dual partner of the diagonal action of $\Cu_{q^2}(\lie{sl}_2)$ on the space of symplectic polynomial-spinors and Corollary \ref{cor:FischerTriang}, which is the appropriate analogue of the metaplectic Fischer triangle in our context. We remark that the quantum parameter of the dual quantum group is $q$, the square root of the parameter of the diagonal quantum group. Furthermore, the dual partner contains a $q$-analogue of the symplectic Dirac operator (for a treatment in the classical setting, see \cite{HH06} and \cite{DBSS14}) and, along the way, we give precise formulas for the generators of the space of symplectic polynomial monogenics for each homogeneous degree $d\geq 0$. We then arrive at Section \ref{s:projection} where we give an ad hoc formulation (bypassing the need of using $q$-analogues of extremal projectors, and hence working in suitable localizations of the algebras involved) of the projection operators $\Cp(\bbs)^d\to \Cm^d$ onto the space of symplectic monogenics for each homogeneous degree (see Theorem \ref{t:Projection} for the precise statement and \eqref{eq:PolySpinDecomp}, Definition \ref{d:Monogenics}). Finally, in Section \ref{s:GenSym}, we give rather explicit descriptions of generalized symmetries of the $q$-symplectic Dirac operator sending $\Cm^{d}\to \Cm^{d\pm 1}$, for all $d\geq 0$.

We stress that the results obtained in the paper are all conditioned to the hypothesis that $q$ is not a root of unity. This assumption is crucial since in virtually all the formulas we consider when describing the joint decomposition of the polynomial-spinor space, division by a quantum integer is involved. However, to define the commuting pair of quantum $\lie{sl}_2$'s, we only need to divide by $[2]_q$. Thus, it may be interesting to study this duality in the case when $q$ is a root of unity outside the case $q^4=1$ and consider the finite-dimensional counterparts of the polynomial spaces, as described in \cite[Section 4.2]{Ha90}. We leave this and the topic of considering higher rank dualities to a future work. For the latter, a first step would be to realize the diagonal representation of the symplectic quantum group in terms of differential operators using \cite{ZH17}. The difficulty lies in the fact that, when interpreting these operators in the $q$-Weyl algebra as in the present work, quartic terms (in terms of $q$-differential and multiplication operators) appear and the computations become significantly more evolved.

\subsection*{Note on the arXiv version}  For the interested reader, the tex file of the arXiv version of this paper includes hidden details of some straightforward computations and arguments that are omitted in the pdf file.  These details can be displayed by switching the \texttt{details} toggle to true in the tex file and recompiling.

\section{Preliminaries}\label{s:Prelim}
A standard reference for this part is \cite[Chapter 2]{KS12}. Let $q\in\mathbb{C}^\times\setminus \{\pm 1\}$. For any $z\in\bbc$, we let $[z]_q=(q^z-q^{-z})/(q-q^{-1})$ denote its quantum analogue. From elementary calculus, 
it is straightforward to check that $\lim_{q\to 1}[z]_q = z$. In the particular case when $n\in \bbz$, the quantum integer $[n]_q$ satisfies
\[
[n]_q = q^{n-1} + q^{n-3}+\cdots +q^{3-n} + q^{1-n},
\]
where we note that $[0]_q = 0$ and $[1]_q = 1$. Note that when $q$ is not a root of unity, we have that $[n]_q\neq 0$, for all $n\in \bbz_{\geq 1}$.
\subsection{Binomial identities}

For $n\in \bbz_{\geq 0}$, the quantized factorial is defined via $[0]_q! = 1$ and $[n]_q!=\prod_{m=1}^n[m]_q!$. It is clear from the definition that
$\prod_{\ell=0}^{m-1} [n-\ell]_q = [n]_q!/[n-m]_q!$, when $0\leq m \leq n$. The quantized binomial coefficients are given by
\[
\sqbinom{n}{m}_q = \frac{[n]_q!}{[m]_q![m-n]_q!}  
\]
and the quantized analogue of the binomial formula reads as 
\begin{equation}\label{e:binomid}
\prod_{\ell=1}^{n}(1+q^{2\ell}x) = \sum_{\ell=0}^n\sqbinom{n}{\ell}_qq^{(n+1)\ell} x^\ell.
\end{equation}
In what follows, we will need another quantized analogue of a known identity involving the binomial coefficients. For that, consider the divided difference operator $\delta_q$ acting on functions $\varphi(x)$ of one variable via 
\[
(\delta_q \varphi)(x) = \frac{\varphi(q^2x) - \varphi(x)}{q^2x - x}.
\]
Note that if $\varphi$ is differentiable, we have $\lim_{q\to 1}(\delta_q \varphi)(x) = \varphi'(x)$. Note also that for the monomials $\varphi(x) = x^m$ we have $\delta_q(x^m) = q^{m-1}[m]_qx^{m-1}$, and,  for $0\leq j \leq m$, we have
\begin{equation}\label{eq:iterateddivdiff}
    \delta_q^j(x^m) = q^{mj - j(j+1)/2}\frac{[m]_q!}{[m-j]_q!}x^{m-j}.
\end{equation}
Furthermore, the operators $\delta_q$ satisfy the twisted Leibniz rule
\begin{equation}\label{eq:twistedLeibniz}
\delta_q(\varphi\psi)(x) = \varphi(q^2x)(\delta_q\psi)(x) + (\delta_q\varphi)(x)\psi(x).
\end{equation}

Now, if we denote the binomial formula by $\beta(x) = \prod_{\ell=1}^{n}(1+q^{2\ell}x)$, it straightforward to check from the definition of $\delta_q$ that $(\delta_q \beta)(x) = q^{n+1}[n]_q\prod_{\ell=2}^{n}(1+q^{2\ell}x).$
Iterating this procedure for $0\leq j \leq n$, we obtain
\begin{equation}\label{eq:diffbinom}
(\delta_q^j \beta)(x) = q^{j(2n+j+1)/2}\frac{[n]_q!}{[n-j]_q!}\left(\prod_{\ell=j+1}^{n}(1+q^{2\ell}x)\right). 
\end{equation}

\begin{lemma}\label{l:BinomId}
    Let $N\geq n$. Then, for all $0\leq j <n$ we have
    \[
    \sum_{\ell=0}^n(-1)^\ell\sqbinom{n}{\ell}_q \sqbinom{N+\ell}{j}_q q^{\ell(j-(n-1))} = 0.
    \]
    In particular, $\sum_{\ell=0}^n(-1)^\ell\sqbinom{n}{\ell}_q \sqbinom{N+\ell}{n-1}_q = 0$.
\end{lemma}

\begin{proof}
    Using \eqref{eq:diffbinom}, note that $(\delta_q^j \beta)(-q^{-2n})=0$ for all $0\leq j <n$. So, from \eqref{eq:twistedLeibniz}, and for all $0 \leq j < n \leq N$, it follows that $\delta^j_q(x^N\beta)(-q^{-2n})=0$. Now, from \eqref{e:binomid} we get
    \[
    x^N\beta(x) = \sum_{\ell=0}^n\sqbinom{n}{\ell}_qq^{(n+1)\ell} x^{N+\ell},
    \]
    so, recalling the expression for $\delta^j_q(x^{N+\ell})$ from \eqref{eq:iterateddivdiff} we compute that
    \[
    \delta^j_q(x^N\beta)(x) = q^{(Nj-j(j+1)/2)}\sum_{\ell=0}^n\sqbinom{n}{\ell}_q\frac{[N+\ell]_q!}{[N+\ell-j]_q!}q^{(n+1+j)\ell} x^{N+\ell-j}.
    \]
    Evaluating at $-q^{-2n}$ and multiplying by $(-1)^{N-j}q^v([j]_q!)^{-1}$ for a suitable exponent $v$, we get
\[
    \sum_{\ell=0}^n(-1)^\ell\sqbinom{n}{\ell}_q \sqbinom{N+\ell}{j}_q q^{(-n+j+1)\ell} = 0,
\]
as required.
\end{proof}

\section{Quantized Weyl algebras}\label{s:qWeyl}
We will consider a normalized version of the divided difference operator $\delta_q$ used in Section \ref{s:Prelim}. Given a function of one variable $\varphi(x)$ the $q$-derivative $\partial_q$ is the divided difference operator
\[
(\partial_q \varphi)(x) = \frac{\varphi(q x)-\varphi(q^{-1}x)}{q x-q^{-1}x}.
\]
 Following \cite{Ha90}, we will now describe $q$-analogues of the algebra of partial differential operators with polynomial coefficients. For each $N\in\mathbb{Z}_{\geq 1}$, let $\mathcal{P}_N:=\mathbb{C}[x_1,\ldots,x_N]$ denote the polynomial space on $N$-variables. Recall that $q\notin \{1,-1\}$.
 
\begin{definition}
    The $N$-th quantized Weyl algebra $\mathcal{W}_q(N)$ is defined as the unital, associative $\mathbb{C}$-algebra generated by
    \begin{align}\label{eq:qweylgen}
        \mu_j,\partial_j,\gamma_j^{\pm 1}, \ \ \  1\leq j \leq N
    \end{align}
    and subject to the following relations for each $1\leq i,j\leq N$:
    \begin{gather}\label{eq:qWeylRelGam}
    \gamma_j\gamma_i=\gamma_i\gamma_j,\quad \gamma_j\gamma_j^{-1}=1,\\\label{eq:qWeylRelGamCom}
    \gamma_j\partial_i\gamma_j^{-1} = q^{-\delta_{ij}}\partial_i,\quad\gamma_j\mu_i\gamma_j^{-1} = q^{\delta_{ij}}\mu_i,\\\label{eq:qWeylRelCom1}
    [\partial_j,\partial_i]=0=[\mu_j,\mu_i],\\
    [\partial_j,\mu_i]=0, \ \ {\rm if} \ \ i\neq j,\\\label{eq:qWeylRelCom2}
    \partial_j\mu_j - q^{\pm 1}\mu_j\partial_j = \gamma_j^{\mp 1}.
    \end{gather}   
    Here, the symbol $\delta_{ij}$ means the Kronecker delta.
\end{definition}
\begin{remark}
For ease of notation it will be convenient to define $\{a\}_q = (a-a^{-1})/(q-q^{-1})$ for any invertible element $a$ in an associative algebra over $\bbc$.
    In particular, note that the last relations in (\ref{eq:qWeylRelCom2}) are equivalent to the relations
\begin{align}\label{eq:qWeylRelCom2dif}
    \mu_j\partial_j = \{\gamma_j\}_q,\quad \partial_j\mu_j = \{q\gamma_j\}_q,
\end{align}
for $1 \leq j \leq N$.
\end{remark}
The relations \eqref{eq:qWeylRelGam}--\eqref{eq:qWeylRelCom2} imply a realization of $\mathcal{W}_q(N)$ inside $\End(\mathcal{P}_N)$ where the action is defined on the monomial basis (in multi-index notation)
\[
\{x^\alpha:=x_1^{a_1}\cdots x_N^{a_N}\mid \alpha=(a_1,\ldots,a_N)\in\mathbb{Z}_{\geq 0}^N\} \subseteq \mathcal{P}_N
\]
via
\[
    \gamma_j^{\pm 1}(x^\alpha) = q^{\pm a_j}x^\alpha,\qquad
    \mu_j(x^\alpha) = x^{\alpha+\varepsilon_j},\qquad
    \partial_j(x^\alpha) = [a_j]_{q}x^{\alpha-\varepsilon_j},
\]
where $\varepsilon_j\in \mathbb{Z}_{\geq 0}^N$ is the vector with $1$ in the $j$-th coordinate and $0$ everywhere else. Note that for each index $1\leq j \leq N$, the partial $q$-derivative operator is explicitly realized as 
\[
    \partial_j = \partial_{j,q} = \frac{1}{x_j}\frac{\gamma_j-\gamma_j^{-1}}{q-q^{-1}}.
\]
Now recall that $[m]_q = q^{m-1}+q^{m-2}+\cdots+q^{2-m}+q^{1-m}$, for any $q\in \bbc^\times$. Formally replacing $q$ by the invertible element $\gamma_j$ yields a well-defined element $[m]_{\gamma_j}\in \Cw_q(N)$ that satisfies
    \[
    (\gamma_j-\gamma_j^{-1})[m]_{\gamma_j} = \gamma_j^{m} - \gamma_j^{-m}.
    \]
    In particular, for all $m\in \bbz_{\geq 1}$, the shifted $q$-derivatives $\partial_{j,q^m}$, given by
\begin{equation}\label{eq:powerdiff}
    \partial_{j,q^m} = \frac{1}{x_j}\frac{\gamma_j^{m}-\gamma_j^{-m}}{q^m-q^{-m}} = \frac{1}{x_j}\frac{(\gamma_j-\gamma_j^{-1})[m]_{\gamma_j}}{(q-q^{-1})[m]_q} = \partial_{j,q}\frac{[m]_{\gamma_j}}{[m]_q},
\end{equation}
will be well-defined elements in $\Cw_q(N)$, provided that $q$ is not a root of unity.

In what follows, given $v\in\bbc^\times$ and elements $A,B$ in some associative algebra over $\bbc$, we will denote their twisted commutator by
\begin{equation}\label{eq:twistedComm}
    [A,B]_v = AB-vBA.
\end{equation} 
For example, note that the relation \eqref{eq:qWeylRelCom2} can be written as $[\partial_j,\mu_j]_{q^{\pm 1}} = \gamma_j^{\mp 1}$. We now summarize some commutation relations in $\Cw_q(N)$ that will be useful later on. 

\begin{proposition}\label{p:qWeylComm}
    The following relations hold, for indices $1\leq i,j \leq N$, $m\in\bbz_{\geq 1}$ and $q\in\bbc^\times$ such that $[m]_q\neq 0$.
    \begin{align}\label{eq:CommId0a}
    [\partial_j,\mu_j^2]_{q^{\pm 2}} &= q^{\pm 1}[2]_q\gamma_j^{\mp 1}\mu_j,\\\label{eq:CommId0b}
    [\partial_j^2,\mu_j]_{q^{\pm 2}} &=[2]_q\gamma_j^{\mp 1}\partial_j,\\
    \label{eq:CommId1}
    [\partial_{j,q^m},\mu_j]_{q^{\pm m}} &= \gamma_j^{\mp m},\\
    \label{eq:CommId2}
    [\mu_i\partial_{j,q^m},\mu_j\partial_{i,q^m}] &= \{(\gamma_i\gamma_j^{-1})^{m}\}_{q^m},\\\label{eq:CommId3}
    [\partial_{i,q^m}\partial_{j,q^{m}},\mu_i\mu_j] &= \{(q\gamma_i\gamma_j)^m\}_{q^m},\\\label{eq:CommId4}
    [\mu_i\partial_j,\mu_j\partial_{i,q^2}]_{q^{\pm 1}} &= \frac{(\pm 1)}{q- q^{-1}}\left(q^{\pm 1}\gamma_i^{\pm 2}\gamma_j^{\mp 1} - \frac{q^{\pm 2}\gamma_i^{\pm 2}\gamma^{\pm 1}_j}{[2]_q} - \frac{\gamma_i^{\mp 2}\gamma_j^{\pm 1}}{[2]_q}\right),\\\label{eq:CommId5}
    [\partial_{i,q^2}\partial_j,\mu_i\mu_j]_{q^{\pm 1}} &= \frac{(\pm 1)}{q- q^{-1}}\left(\frac{q^{\pm 2}\gamma_i^{\pm 2}\gamma_j^{\pm 1}}{[2]_q} + \frac{\gamma_i^{\mp 2}\gamma_j^{\pm 1}}{[2]_q} -  q^{\mp 1}\gamma_i^{\mp 2}\bar\gamma_j^{\mp 1}\right).
    \end{align}
\end{proposition}

\begin{proof}
    These relations are obtained by straightforward computations using \eqref{eq:qWeylRelGam}--\eqref{eq:powerdiff}.
    \details{ 

    For the first identity, we compute directly using \eqref{eq:qWeylRelCom2dif} that
    \[
    [\partial_j,\mu_j^2]_{q^{\pm 2}} = \{q\gamma_j\}_q\mu_j - q^{\pm 2}\mu_j\{\gamma_j\}_q
    = \{q\gamma_j\}_q\mu_j - q^{\pm 2}\{q^{-1}\gamma_j\}_q\mu_j
    = q^{\pm 1}[2]_q\gamma_j^{\mp 1}\nu.
    \]
    The second identity is derived in an entirely analogous fashion.
    Note that the third identity for $m=1$ is \eqref{eq:qWeylRelCom2}. For $m>1$, using \eqref{eq:powerdiff},  we have
    \begin{equation}\label{eq:genProb}
    \partial_{j,q^m}\mu_j = \frac{1}{x_j}\frac{\gamma_j^m-\gamma_j^{-m}}{q^m-q^{-m}}\mu_j = \{q^m\gamma_j^m\}_{q^m},  
    \end{equation}
     and hence $[\partial_{j,q^m},\mu_j]_{q^{\pm m}} = \{q^m\gamma_j^m\}_{q^m} - q^{\pm m}\{\gamma_j^m\}_{q^m}=\gamma_j^{\mp m}$, proving \eqref{eq:CommId1}. For the forth identity, using \eqref{eq:genProb} we compute
    \[
    [\mu_i\partial_{j,q^m},\mu_j\partial_{i,q^m}]=\{\gamma_i^m\}_{q^m}\{q^m\gamma_j^m\}_{q^m}-\{q^m\gamma_i^m\}_{q^m}\{\gamma_j^m\}_{q^m} = \{\gamma_i^m\gamma_j^{-m}\}_{q^m},
    \]
    while for the fifth identity we have
    \[
    [\partial_{i,q^m}\partial_{j,q^{m}},\mu_i\mu_j]=\{q^m\gamma_i^m\}_{q^m}\{q^m\gamma_i^m\}_{q^m} - \{\gamma_i^m\}_{q^m}\{\gamma_j^m\}_{q^m} = \{q^m\gamma_i^m\gamma_j^m\}_{q^m}.
    \]
    For the last two relations, we first compute
    \begin{align*}
     [\mu_i\partial_j,\mu_j\partial_{i,q^2}]_{q} &= \{\gamma_i^2\}_{q^2}\{q\gamma_j\}_q - q\{q^2\gamma_i^2\}_{q^2}\{\gamma_j\}_q\\
     &=\frac{[2]_q^{-1}}{(q-q^{-1})^2}\left((\gamma_i^2-\gamma_i^{-2})(q\gamma_j-(q\gamma_j)^{-1})-q((q\gamma_i)^2-(q\gamma_i)^{-2})(\gamma_j-\gamma_j^{-1})
     \right)\\
     &=\frac{1}{q- q^{-1}}\left(q\gamma_i^{2}\gamma_j^{-1} - \frac{q^{2}\gamma_i^{2}\gamma_j}{[2]_q} - \frac{\gamma_i^{- 2}\gamma_j}{[2]_q}\right).
    \end{align*}
    The other three cases are similar.}
\end{proof}

\section{The quantized enveloping algebra of \texorpdfstring{$\lie{sl}_2$}{sl2}}\label{s:Uqsl2}
Given $q\in \mathbb C^\times$ recall that the quantized universal enveloping algebra $\mathcal U_q(\mathfrak{sl}_2)$ is the associative algebra (with $1$) generated by $e,f,k^{\pm 1}$, subject to the following relations:
\begin{gather}kk^{-1} = 1 = k^{-1}k, \label{krel}\\
kek^{-1} = q^2e, \ \ \ kfk^{-1} = q^{-2}f, \label{wrel}\\
ef-fe = \dfrac{k-k^{-1}}{q-q^{-1}}. \label{efrel} 
\end{gather}

There is a Hopf algebra structure on $\mathcal U_q(\mathfrak{sl}_2)$ where the comultiplication $\Delta$, antipode $S$ and counit $\epsilon$ are given by 
\begin{gather*}
    \Delta(e) = e\otimes k + 1\otimes e, \ \ \Delta(f) = f\otimes 1 + k^{-1}\otimes f, \ \ 
    \Delta(k^{\pm 1})= k^{\pm 1}\otimes k^{\pm 1},\\
    S(e)= -ek^{-1}, \ \ S(f) = -kf, \ \ S(k^{\pm 1})=k^{\mp 1},\\
       \epsilon(e)= 0=\epsilon(f), \ \ \epsilon(k^{\pm 1})=1. \end{gather*}
       
We abbreviate $\mathcal U_q=\mathcal U_q(\lie{sl}_2)$ and denote by $\mathcal U_q^-$, $\mathcal U_q^0$, and $\mathcal U_q^+$ the subalgebras of $\Cu_q$ generated by $e$, $k^{\pm 1}$ and $f$, respectively. Moreover, the multiplication establishes an isomorphism of vector spaces
\begin{equation}\label{e:triangular}
    \mathcal U_q \cong \mathcal U_q^+\otimes \mathcal U_q^0\otimes \mathcal U_q^-.
\end{equation}

Recalling the notation $\{a\}_q = (a-a^{-1})/(q-q^{-1})$ for any invertible element $a$ in an associative algebra, note that, in $\Cu_q$ we have
\[
[e,f] = \frac{k-k^{-1}}{q-q^{-1}} = \{k\}_q.
\]

\begin{lemma}\label{l:FEcommutations}
    Given $0\leq r\leq d$, the following identities hold in $\Cu_{q}$:
    \begin{gather}
        [e^d,f] = e^{d-1}[d]_q\{q^{d-1}k\}_q, \label{FErel}\\
    f^re^d \equiv (-1)^r\frac{[d]_q!}{[d-r]_q!}e^{d-r}\prod_{\ell=1}^r\{q^{d-\ell}k\}_q\mod \Cu_q f. \label{e:FEcommut}
    \end{gather}
\end{lemma}
\begin{proof}
    Equation \eqref{FErel} is a straightforward computation using the relation $[e,f] = \{k\}$. 
    \details{ In fact, we have
    \begin{align*}
        [e^d,f] &= \sum_{r=1}^d e^{d-r}\{k\}_qe^{r-1}\\
        &=e^{d-1}\left(\sum_{r=1}^d \{q^{2r-2}k\}_q\right)\\
        &=\frac{e^{d-1}}{q-q^{-1}}\left(k\frac{(q^{2d}-1)}{(q^2-1)} - k^{-1}\frac{(1-q^{-2d})}{(1-q^{-2})}\right)\\
        &=\frac{e^{d-1}}{q-q^{-1}}\left(q^{d-1}k[d]_q - (q^{(d-1)}k)^{-1}[d]_q\right)\\
        &=[d]_qe^{d-1}\{q^{d-1}k\}_q.
    \end{align*}} 
    To prove \eqref{e:FEcommut} first note it clearly holds for $r=0$ and hence, from now on we assume $r\geq 1$.  Since
    $$f^r e^d = e^df^r + [f^r,e^d] \equiv [f^r,e^d] \mod \Cu_q f,$$
    showing \eqref{e:FEcommut} is equivalent to show that for $1\leq r\leq d$ we have 
    \begin{equation*}\label{e:commFE}[f^r,e^d] \equiv (-1)^re^{d-r}\prod_{\ell=1}^r[d-\ell+1]_q\{q^{d-\ell}k\}_q\  \mod \Cu_q f.\end{equation*}
    We proceed by induction on $r$ which begins when $r=1$ by \eqref{FErel}. For the inductive step suppose $1<r<d$. Therefore, using the inductive hypothesis we have 
    $$[f^{r+1}, e^d] = f[f^r, e^d] + [f^r, e^d]f \equiv (-1)^rfe^{d-r}\prod_{\ell=1}^r[d-\ell+1]_q\{q^{d-\ell}k\}_q \ \mod \Cu_q f,$$
    and hence, a further application of \eqref{FErel} gives, modulo $\Cu_q f$, that 
    \begin{align*}[f^{r+1},e^d] &\equiv (-1)^{r}(-e^{d-r-1}[d-r]_q\{q^{d-r-1}_qk\})\prod_{\ell=1}^r[d-\ell+1]_q\{q^{d-\ell}k\}_q\\
         &\equiv (-1)^{r+1}e^{d-(r+1)}\prod_{\ell=1}^{r+1}[d-\ell+1]_q\{q^{d-\ell}k\}_q,
    \end{align*}
    which completes the proof of the lemma. 
\end{proof}

\subsection{Adjoint actions}
Given $x\in\Cu_q$, we use Sweedler's notation and write $\Delta(x) = \sum{x_{(1)}}\otimes x_{(2)}$. Using the Hopf algebra structure, the left adjoint action of $\Cu_q$ on itself (see \cite{BNP09} and the references therein) is defined by
\[
\ad_x(a) = \sum x_{(1)}aS(x_{(2)}).
\]
In particular we get the formulas
\begin{equation}\label{eq:adjointact}
\begin{aligned}
    \ad_e(a) &= eak^{-1} - aek^{-1}&\qquad
    \ad_f(a) &= fa - k^{-1}akf\\
    \ad_k(a) &= kak^{-1}&\qquad
    \ad_{k^{-1}}(a) &= k^{-1}ak.
\end{aligned}
\end{equation}
More generally, let $\Ca$ be any associative algebra and $\sigma:\Cu_q\to \Ca$ a homomorphism of algebras.
For any $x\in\Cu_q$ and $a\in \Ca$, the formula
\[
    \ad^\sigma_x(a) = \sum \sigma(x_{(1)})a\sigma(S(x_{(2)}))
\]
defines a representation of $\Cu_q$ on $\Ca$, which will  be referred to as the adjoint representation of $\Cu_{q}$ on $\Ca$. To ease notation, this representation may sometimes be denoted only by $\ad_x(a)$ instead of $\ad_x^\sigma(a)$.

\subsection{Representation theory}

Given a $\mathcal U_q$-module $V$, recall that $V$ is a weight module if 
$$V= \bigoplus_{\mu\in \mathbb K^\times} V_\mu, \ \ V_\mu = \{v\in V : kv = \mu v\}.$$
Whenever $V_\mu\neq 0$ we say that $\mu$ is a weight of $V$ and $V_\mu$ is the associated weight space of weight $\mu$. Every nonzero vector of $V_\mu$ is called a weight vector of weight $\mu$. Moreover, for $\mu\in \mathbb C$ we have
$$eV_\mu \subset V_{\mu q^2} \ \ \ {\rm and} \ \ \ fV_\mu \subset V_{\mu q^{-2}}.$$
If $v$ is weight vector and $fv=0$, then $v$ is said to be a lowest-weight vector. In light of \eqref{e:triangular}, if $V$ is generated by a lowest-weight vector $v$ then $V=\mathcal U_q^+v$ and $V$ is said to be a lowest weight module.\\ 

Given $\lambda\in \mathbb C^\times$ let $M_q(\lambda)$ be the quotient of $\mathcal U_q$ by left ideal generated by $f$ and $k-\lambda$. Clearly $M_q(\lambda)$ is a lowest weight module of weight $\lambda$ and $M_q(\lambda)$ has a unique irreducible quotient $L_q(\lambda)$. The following result is well known (see, for instance \cite{CP94}).
\begin{proposition} 
\begin{enumerate}\label{p:irredclas}
    \item [(i)] The module $M_q(\lambda)$ is reducible if and only if $\lambda = \pm q^{r}$, for some $r\in \mathbb Z_{\leq 0}$. In this case we have $\dim L_q({\pm q^{r}})=-r+1$. 
    \item[(ii)] If $V$ is a finite dimensional irreducible $\mathcal U_q$--module then $V\cong L_q(\sigma q^r)$ for some $\sigma\in\{-1,1\}$ and $r\in \mathbb Z_{\leq 0}$. \qed
\end{enumerate}
\end{proposition}
The next result on the decomposition of the tensor product when the weights are not integer powers of the quantum parameter is well known and the proof follows the same lines of its classical version. We provide a proof for convenience of the reader. Note that in the case when the weights are integral the situation becomes more complicated (see \cite{CMJ}).
\begin{lemma}\label{l:tpdecomp}
    Let $r\in \mathbb Z_{\leq 0}$ and  $\lambda \in \mathbb C^\times\setminus \{q^{\mathbb Z}\}$. Then 
    $$L_q(\lambda)\otimes L_q(q^r)\cong \bigoplus_{\ell=0}^{-r} L_q(\lambda q^{r+2\ell}).$$
\end{lemma}
\begin{proof}
Let $M=L_q(\lambda)\otimes L_q(q^r)$ and note that for $\ell\in \mathbb Z_{\geq 0}$ we have $\dim M_{\lambda q^{r+2\ell}}= \min\{\ell,-r\}+1$. 
In particular, for each $0\leq \ell\leq -r$ there exists a non zero vector $v_\ell\in M$ such that 
$$v_\ell \in M_{\lambda q^{r+2\ell}}\ \ \ {\rm and} \ \ \ f v_\ell=0.$$
Moreover, since $\lambda q^{r+2\ell}\notin q^{\mathbb Z}$, it follows from Proposition \ref{p:irredclas} that $\Cu_q v_\ell \cong L_q(\lambda q^{r+2\ell})$ for all $0\leq \ell\leq -r$. 
Using the fact that $L_q(\lambda q^{r+2\ell})$ is irreducible and its lowest weights are distinct for each $\ell$, we get 
$$L:=\bigoplus_{\ell=0}^{-r}L_q(\lambda q^{r+2\ell})\subseteq M,$$
and hence the lemma follows by noting that $\dim L_\mu = \dim M_{\mu}$ for all $\mu\in \mathbb C^\times$. 
\end{proof}

\begin{remark}
    Throughout this paper we shall be particularly interested in irreducible representations $L_q(\lambda)$ such that $\lambda = q^m$ with $m\in \frac12\mathbb Z$. Henceforth we shall abuse notation and simply write $L_q(m)$ for $L_q(q^m)$. 
\end{remark}

\section{Quantum metaplectic duality}\label{s:qDuality}

From this section onward, we assume that $q\in\bbc^\times$ is not a root of unity. Note that this assumption implies that all quantum integers $[m]_q$ will be nonzero. We start by describing a quantum analogue of the representation of $\mathfrak{sl}_2$ in the symmetric powers of its fundamental representation in terms of $q$-differential operators. We consider a particular situation of the $q$-coordinate rings discussed in \cite{ZH17} and in \cite[Chapter 9]{KS12}. 
    
    \begin{proposition}\label{p:natrep}
        The assignment $\sigma_{w}:\mathcal{U}_{q^2} \to \mathcal{W}_{q}(2)\subset \End(\mathcal{P}_2)$
        \begin{align*}
    \sigma_{w}(k) := K_w &= \gamma_1^2\gamma_2^{-2},\\
    \sigma_{w}(e) := E_w &= \mu_1\partial_{2,q^2},\\
    \sigma_{w}(f) := F_w &= \mu_2\partial_{1,q^2},
\end{align*}
defines a representation of $\mathcal{U}_{q^2}$ on $\mathcal{P}_2$. Moreover, for each $d\in \mathbb Z_{\geq 0}$ we have  
$$\mathcal P_2^d  \cong_{\mathcal U_{q^2}} L_{q^2}(-d).$$
    \end{proposition}
    \begin{proof}
        First note that 
        \[
        K_wE_w = \gamma_1^2\gamma_2^{-2}\mu_1\partial_{2,q^2} = (\gamma_1^2\mu_1)(\gamma_2^{-2}\partial_{2,q^2}) = q^4(\mu_1\gamma_1^2)(\partial_{2,q^2}\gamma_2^{-2}) = q^4E_wK_w.
        \]
        Similarly $K_wF_w = q^{-4}F_wK_w$. Next, using \eqref{eq:CommId2}, we compute
        \[
            E_wF_w - F_wE_w = [\mu_1\partial_{2,q^2},\mu_2\partial_{1,q^2}]
            =\{K_w\}_{q^2},
        \]
        as required for the first statement. For the moreover part, it is clear from the definition of $\sigma_w$ that $\mathcal P_2^d$ is a $\mathcal U_{q^2}$-submodule of $\mathcal P_2$. Further, since 
       $$F_w (x_2^d) = 0, \ \ K_w (x_2^d) = q^{-2d}x^d,\ \ {\rm and} \ \  \dim \mathcal P^d = d+1,$$
       it follows from Proposition \ref{p:irredclas} that $\mathcal P_2^d\cong L_{q^2}(-d)$.
       \end{proof}

Next we consider an action of  $\mathcal U_{q^2}$ on $\mathcal P_1$ through the quantum Weyl algebra $\mathcal W_q(1)$. To avoid confusion, we shall use a different notation in this case and write $\mathcal C_{s,q} = \mathcal W_q(1)$, generated by $\nu,\nabla,\omega^{\pm 1}$, instead of $\mu_1,\partial_1,\gamma_1^{\pm {1}}$, respectively, c.f. \eqref{eq:qweylgen}. The algebra $\Cc_{s,q}$ will be referred to as the quantized symplectic Clifford algebra. Further, we write $\mathbb S = \mathbb K[y]$, instead of $\Cp_1$, and we have $\mathbb S = \mathbb S_+\oplus \mathbb S_-$, with 
$$\mathbb S_+ = \mathbb K[y^2], \ \ \ \mathbb S_- = y\mathbb K[y^2].$$ 
The space $\bbs$ is called the space of symplectic spinors. The representation of $\Cc_{s,q}$ on $\bbs$ below is often referred to as the $q$-oscillator representation.

\begin{proposition}\label{p:spinrep}
        The assignment $\sigma_{s}:\mathcal{U}_{q^2} \to \mathcal{C}_{s,q}\subset \End(\mathbb S)$
        \begin{align*}
    \sigma_{s}(k) := K_s &= q\omega^2,\\
    \sigma_{s}(e) := E_s &= \frac{1}{[2]_q}\nu^2,\\
    \sigma_{s}(f) := F_s &= -\frac{1}{[2]_q}\nabla^2,
\end{align*}
defines a representation of $\mathcal{U}_{q^2}$ on $\mathbb S$. Moreover,
$$\mathbb S_+ \cong_{\mathcal U_{q^2}}L_{q^2}\left(\tfrac12\right), \ \ \ \mathbb S_- \cong_{\mathcal U_{q^2}}L_{q^2}\left(\tfrac32\right).$$
\begin{proof}
    The realization of the $q$-oscillator representation was discussed in \cite[p. 229]{NUW96}. For the moreover part it is clear from the definition of $\sigma_s$ that $\mathbb S_{\pm}$ are $\mathcal U_{q^2}$-submodules of $\mathbb S$. The result then follows from Proposition \ref{p:irredclas} by noting that
    \begin{gather*}
        F_s (1) = 0,\  \ \ K_s(1) = q = (q^2)^{\frac12}, \ \ \ \sigma_s(\mathcal U_{q^2}^+)(1) = \bbs_+,\\
        F_s (y) = 0,\  \ \ K_s(y) = q^3y = (q^2)^{\frac32}y, \ \ \ \sigma_s(\mathcal U_{q^2}^+)(y) = \bbs_-.
    \end{gather*}
    This finishes the proof.
\end{proof}
\end{proposition}

The tensor product of these two representations $\Cp(\bbs)=\Cp_2\otimes \bbs$ will be called the space of polynomial symplectic spinors.
With respect to the homogeneous degrees of $\mathcal P_2$, we have the decomposition 
\begin{equation}\label{eq:PolySpinDecomp}
\mathcal P(\mathbb S) =\bigoplus_{d\geq 0}\mathcal P(\mathbb S)^d,\ \ \ {\rm with} \ \ \ \mathcal P(\mathbb S)^d= \mathcal P_2^d\otimes \mathbb S.
\end{equation}
Set also $\mathcal P(\mathbb S)^d_\pm = \mathcal P_2^d\otimes \mathbb S_{\pm}$.  The following result is straightforward from the Hopf algebra structure of $\mathcal{U}_{q^2}$, Propositions \ref{p:natrep} and \ref{p:spinrep}, and Lemma \ref{l:tpdecomp}. 
\begin{proposition}\label{p:deltaact}
    The assignment $\sigma_{\Delta}:\mathcal{U}_{q^2}\to \mathcal{W}_{q}(2)\otimes\mathcal{C}_{s,q}\subseteq \End(\Cp(\bbs))$ given by
    \begin{align*}
        \sigma_{\Delta}(k) = \sigma_{w}(k)\otimes \sigma_{s}(k) := K_{\Delta} &= q\gamma_1^2\gamma_2^{-2}\omega^2,\\
        \sigma_{\Delta}(e) = \sigma_{s}(e)\otimes \sigma_{w}(k) + 1\otimes\sigma_{s}(e) := E_{\Delta} &= q\omega^2\mu_1\partial_{2,q^2}+\frac{\nu^2}{[2]_q},\\
        \sigma_{\Delta}(f) = \sigma_{w}(f)\otimes 1 + \sigma_{w}(k^{-1})\otimes \sigma_{s}(f) := F_{\Delta} &= \mu_2\partial_{1,q^2}-\frac{\gamma_1^{-2}\gamma_2^{2}}{[2]_q}\nabla^2,
    \end{align*}
    extends to a representation of $\mathcal{U}_{q^2}$ on $\mathcal P(\mathbb S)$. Moreover, as $\Cu_{q^2}$-modules we have 
    \begin{gather*}\mathcal P(\bbs)\cong \bigoplus_{d\geq 0}\left(\mathcal P(\bbs)^d_+\oplus \mathcal P(\bbs)^d_-\right),\ \ \ {\rm and}\\
    \Cp(\bbs)^d_+ \cong_{\Cu_{q^2}} L_{q^2}(\tfrac{1}{2}-d) \oplus L_{q^2}(\tfrac{1}{2}-d+2) \oplus \cdots \oplus L_{q^2}(\tfrac{1}{2}+d), \\
        \Cp(\bbs)^d_- \cong_{\Cu_{q^2}} L_{q^2}(\tfrac{3}{2}-d) \oplus L_{q^2}(\tfrac{3}{2}-d+2) \oplus \cdots \oplus L_{q^2}(\tfrac{3}{2}+d).\hfil\qed
    \end{gather*} 
\end{proposition}

Now that we have established the diagonal action $\sigma_\Delta$ of $\Cu_{q^2}$ on $\Cp(\bbs)$, the next step will be to determine a quantum analogue of its dual pair. In what follows, inspired by the classical duality \cite{DBSS14,DBHS17}, we make an educated guess and search for elements $E, F$ of the type
\begin{equation}\label{eq:DualPairAnsatz}
\begin{aligned}
    E &= c_1\mu_1\nabla + c_2\mu_2\nu,\\
    F &= c_3\partial_{1,q^2}\nu + c_4\partial_{2,q^2}\nabla,\\
\end{aligned}
\end{equation}
where the coefficients $c_{\ell}$ for $\ell=1,2,3,4$ are Laurent polynomials $c_\ell = c_\ell(\gamma_1^\pm,\gamma_2^{\pm},\omega^{\pm})$ on the degree variables of the $q$-Weyl algebra. We will, in fact, suppose that each $c_\ell$ is a monomial of the type $c_\ell = p_\ell\gamma_1^{x_\ell}\gamma_2^{y_\ell}\omega^{z_\ell}$ for exponents $x_\ell,y_\ell,z_\ell\in \bbz$ and for a polynomial $p_\ell$ on $q,q^{-1}$. We remark that in $\Cw_q(2)$ the shifted $q$-derivative $\partial_{j,q^2}$ does not commute with $\partial_j$, so the best hope in finding differential operators commuting with $\sigma_\Delta$ is to consider the shifted $q$-derivatives as in the expression for $F$ in \eqref{eq:DualPairAnsatz}. If we impose that the elements in \eqref{eq:DualPairAnsatz} commute with $\sigma_\Delta$ and solving the equations for the variables that we have introduced, we obtain elements in $\Cw_q(2)\otimes\Cc_{s,q}$ whose commutator is the correct analogue for the degree operator discussed in \cite{DBSS14} (although we have normalized the $\lie{sl}_2$-triple so that it is more amenable to the quantization process). This leads to the following result.

\begin{theorem}\label{thm:spinact}
    The assignment $\sigma:\mathcal{U}_{q} \to \mathcal{W}_{q}(2)\otimes\mathcal{C}_{s}$ given by
    \begin{align*}
    \sigma(k) := K &= q^2\gamma_1^2\gamma_2^2,\\
    \sigma(e) := E &= [2]_q(q^2\gamma_2^2\omega\mu_1\nabla + \mu_2\nu),\\
    \sigma(f) := F &= \partial_{1,q^2}\nu - \gamma_1^{-2}\omega\partial_{2,q^2}\nabla,
    \end{align*}
    extends to a representation of $\mathcal{U}_{q}$ on $\mathcal{P}_2\otimes\bbs$ which commutes with the representation of $\mathcal{U}_{q^2}$ generated by $E_\Delta,F_\Delta$. 
\end{theorem}
\begin{proof}
    The proof will follow by using the relations in $\Cw_q(2)\otimes\Cc_{s,q}$ discussed in Section \ref{s:qWeyl}. We will sketch some of the computations. The crucial observations are that $q$-derivatives and multiplication operators on different directions commute in the $q$-Weyl algebra and that conjugation of a multiplication or $q$-derivative by the degree variables $\gamma_1,\gamma_2,\omega$ yield the same operator, but multiplied by a suitable power of $q$. That said, note that
    \begin{align*}
        [F_\Delta,E] &= [2]_q[\mu_2\partial_{1,q^2}, q^2\gamma_2^2\omega\mu_1\nabla] - [\gamma_1^{-2}\gamma_2^{2}\nabla^2,\mu_2\nu]\\
        &= [2]_q\gamma_2^2\omega\mu_2\nabla[\partial_{1,q^2},\mu_1]_{q^2}-\gamma_1^{-2}\gamma_2^{2}\mu_2[\nabla^2,\nu]_{q^{-2}}\\
        &=0,
    \end{align*}
    where we used \eqref{eq:CommId0b} and \eqref{eq:CommId1} in the last line. 
    Proving that $[F_\Delta,F]=0=[E_\Delta,E]=[E_\Delta,F]$ is done analogously and are left to the reader.
    Finally, we compute $[E,F] = B_1 - B_2$ with
    \begin{align*}
        B_1 &:= [2]_q[q^2\gamma_2^{2}\omega\mu_1\nabla,\partial_{1,q^2}\nu ],\\
        B_2 &:= [2]_q[\mu_2\nu,\gamma_1^{-2}\omega\partial_{2,q^2}\nabla].
    \end{align*}
    Using \eqref{eq:CommId4}, we get
    \begin{align*}
        B_1 &= q^{2}[2]_q\gamma_2^{2}\omega[\mu_1\nabla,\partial_{1,q^2}\nu]_{q^{-1}}\\
        &=\gamma_2^2\omega\frac{[2]_qq^{2}}{q-q^{-1}}\left(\frac{q^{-2}\gamma_1^{-2}\omega^{-1}}{[2]_q} + \frac{\gamma_1^{2}\omega^{-1}}{[2]_q}-q^{-1}\gamma_1^{-2}\omega\right)
    \end{align*}
    and, using \eqref{eq:CommId5},
    \begin{align*}
        B_2 &= [2]_q\gamma_1^{-2}\omega(-[\partial_{2,q^2}\nabla,\mu_2\nu]_{q^{-1}})\\
        &= \gamma_1^{-2}\omega\frac{[2]_q}{q-q^{-1}}\left(-q\gamma_2^{2}\omega+\frac{q^{-2}\gamma_2^{-2}\omega^{-1}}{[2]_q}+\frac{\gamma_2^{2}\omega^{-1}}{[2]_q}\right),
    \end{align*}
    so that $[E,F] = B_1-B_2 = \{(q\gamma_1\gamma_2)^2\}_q$. This finishes the proof.
\end{proof}

\begin{remark}
    The operator $F\in \sigma(\Cu_q)$ is a quantized analogue of the symplectic Dirac operator (in the special case of a symplectic vector space of dimension $2$) discussed in \cite{DBHS17,DBSS14} and which was studied further in \cite{EM24}. We remark that the definition of $F$ given here depended on the specific choice of comultiplication used for $\Cu_{q^2}$ in Proposition \ref{p:deltaact}. Different choices would lead to essentially the same results.
\end{remark}

\subsection{Fischer Triangle}
Now that we have established a commuting pair $(\sigma_\Delta,\sigma)$ of representations of the pair $(\Cu_{q^2},\Cu_q)$ of quantum groups on the space $\Cp(\bbs)$, we shall describe its joint decomposition. 

Just like in the classical theory, we will focus on the kernel of the operator $F$, our $q$-analogue of the symplectic Dirac operator.
Given a typical monomial $\mathbf{m}=x_1^ax_2^b\otimes y^c \in \Cp(\bbs)^d$, note that
\[
K_\Delta(\mathbf{m}) = (q^2)^{a-b+c+\frac12}\mathbf{m}.
\]
For later use, it will be convenient to know the $K_\Delta$-weight spaces $\Cp(\bbs)^d_\lambda$ for the weights $\lambda=(q^{2})^{1/2-d}$ and $\lambda=(q^{2})^{3/2+d}$. It will also be important to know its intersection with the kernel of $F$ in these spaces.
\begin{proposition}\label{p:MonogWtSpaces}
    We have that $\Cp(\bbs)^d_{(q^2)^{1/2-d}} = \textup{span}\{x_2^d\otimes 1\}$ and
    \[
    \Cp(\bbs)^d_{(q^2)^{3/2+d}} = \textup{span}\{x_1^{d-b}x_2^b\otimes y^{2b+1}\mid 0\leq b\leq d\}.
    \]
\end{proposition}

\begin{proof}
    For the weight $\lambda=(q^{2})^{1/2-d}$, the equation $a-b+c+\tfrac{1}{2} = \tfrac{1}{2} - d$, subject to the constraints $a+b=d$ and $a,b,c,d\in \bbz_{\geq 0}$, has a unique solution $(a,b,c)=(0,d,0)$. On the other hand, the equation $a-b+c+\tfrac{1}{2} = \tfrac{3}{2} + d$ with same constraints $a+b=d$ and $a,b,c,d\in \bbz_{\geq 0}$ has solutions $\{(a,b,c) = (d-b,b,2b+1)\mid 0\leq b \leq d\}$. This finishes the proof.
\end{proof}

\begin{proposition}\label{p:SpMonog}
Let $\mathbf{p}_+ = x_2^d\otimes 1$ and
\[
\mathbf{p}_- = \sum_{b=0}^d\sqbinom{d}{b}_{q^2}\frac{q^{2b(d-b-1)}}{[2b+1]_q!!}x_1^{d-b}x_2^b\otimes y^{2b+1}.
\]
Then, $\ker(F) \cap \Cp(\bbs)^d_{(q^2)^{1/2-d}} = \bbc \mathbf{p}_+$ and $\ker(F) \cap \Cp(\bbs)^d_{(q^2)^{3/2+d}} = \bbc \mathbf{p}_-$. Furthermore, we have $F_\Delta(\mathbf{p}_\pm)=0$.
\end{proposition}

\begin{proof}
    It is straightforward to check that $F(\mathbf{p}_+)=0$. From Proposition \ref{p:MonogWtSpaces}, let
    \[
    \mathbf{p} = \sum_{b=0}^d c_b x_1^{d-b}x_2^b\otimes y^{2b+1}
    \]
    be a general element in $\Cp(\bbs)^d_{(q^2)^{3/2+d}}$. 
    Solving the equation $F(\mathbf{p})=0$ leads to the recurrence relation
    \begin{equation}\label{eq:recursion}
       c_{b} = \frac{[d-b+1]_{q^2}}{[b]_{q^2}[2b+1]_q}q^{2d-4b}c_{b-1}, 
    \end{equation}
    for $1\leq b\leq d$. Proceeding by induction, it follows that $\ker(F) \cap \Cp(\bbs)^d_{(q^2)^{3/2+d}} = \bbc \mathbf{p}_-$. The last statement is straightforward computation using the definition of $F_\Delta$ in Proposition \ref{p:deltaact} and the expressions for $\mathbf{p}_\pm$.
    \details{
    Note that $c_1 = c_0q^{2(d-2)}[d]_{q^{2}}/[3]_q$. Supposing, by induction, that 
    \[
    c_{b-1}=\sqbinom{d}{b-1}_{q^2}\frac{q^{2(b-1)(d-b)}}{[2b-1]_q!!}c_0,
    \]
    using the recursion \eqref{eq:recursion} we obtain
    \[
        c_b = \frac{[d-b+1]_{q^2}q^{2d-4b}}{[b]_{q^2}[2b+1]_q}\sqbinom{d}{b-1}_{q^2}\frac{q^{2(b-1)(d-b)}}{[2b-1]_q!!}c_0
        = \sqbinom{d}{b}_{q^2}\frac{q^{2b(d-b-1)}}{[2b+1]_q!!}c_0.
    \]
    Finally, as $F_{\Delta} = \mu_2\partial_{1,q^2}-\frac{\gamma_1^{-2}\gamma_2^{2}}{[2]_q}\nabla^2$, clearly $F_\Delta(\mathbf{p}_+)=0$ and
    \begin{align*}
        F_\Delta(\mathbf{p}_-)&=
        \sum_{b=0}^{d}\sqbinom{d}{b}_{q^2}\frac{q^{2b(d-b-1)}}{[2b+1]_q!!}\left([d-b]_{q^2}x_1^{d-b-1}x_2^{b+1}\otimes y^{2b+1}- \frac{[2b+1]_q[2b]_q}{[2]_qq^{2d-4b}}x_1^{d-b}x_2^b\otimes y^{2b-1}\right)\\
        &=\sum_{b=0}^{d-1}\frac{[d]_{q^2}!}{[b]_{q^2}![d-b-1]_{q^2}!}\left(q^{2b(d-b-1)}-q^{2(b+1)(d-b)-2d}\right)x_1^{d-b-1}x_2^{b+1}\otimes y^{2b+1}\\
        &=0,
    \end{align*}
    where we used that $\tfrac{[2b]_q}{[2]_q}=[b]_{q^2}$.
    }

    \end{proof}

\begin{definition}\label{d:Monogenics}
 For each integer $d\geq 0$, let $\Cm^d = \ker(F)\cap \Cp(\bbs)^d$. These spaces are called the {\it polynomial symplectic monogenics in degree $d$}. We write $\Cm = \oplus_{d\geq 0} \Cm^d$ and call the space $\Cm$ as the space of all {\it polynomial symplectic monogenics}.   
\end{definition}

\begin{proposition}\label{p:Einject}
    The element $E$ acts injectively on $\mathcal P^d(\bbs)$, for all $d\geq 0$. 
\end{proposition}
\begin{proof}
    Let $\mathbf{p} = \sum_{j\geq 0} p_j\otimes y^j$ be a generic element in $\mathcal P(\bbs)^d$, where the polynomial coefficients $p_j\in \mathcal P_2^d$ for all $j\geq 0$. Assume that $E(\mathbf{p})=0$ and suppose that there exists $N\in\mathbb Z_{\geq 0}$ maximal such that $p_N\neq 0$. Since
    $$0=E(\mathbf{p})= \sum_{j\geq 0} p_j\left(a_j x_1\otimes y^{j-1} + q^{-j}x_2\otimes y^{j+1}\right), \ \ \ {\rm for \ some} \ \ a_j\in \mathbb C^\times,$$
    we are forced to have $p_N=0$ and hence a contradiction.  
\end{proof}

\begin{theorem}\label{thm:TriangularDecomp}
    For each integer $d\geq 0$, the $\Cu_{q^2}$--module $\mathcal M^d$ decomposes as $\mathcal M^d = \mathcal M^d_+\oplus \mathcal M^d_-$, with 
    $$\mathcal M^d_+ \cong L_{q^2}(\tfrac12 -d), \ \ \ \ \mathcal M^d_-\cong L_{q^2}(\tfrac32+d).$$
    Furthermore, 
    \begin{equation}\label{eq:HomogDecomp}
        \mathcal P(\bbs)^d = \bigoplus_{k=0}^dE^k(\mathcal M^{d-k}).
    \end{equation}
\end{theorem}
\begin{proof}
We proceed by induction on $d$ which clearly holds for $d=0$, by Proposition \ref{p:deltaact}. Assuming $d>0$, Theorem \ref{thm:spinact} and Proposition \ref{p:Einject} imply that $E$ induces an isomorphism of $\mathcal P(\bbs)^{d-1}$ onto a proper subspace of $\mathcal P(\bbs)^d$. Using Proposition \ref{p:deltaact}, it is straightforward to check that, as $\Cu_{q^2}$--modules, we have
\begin{align*}
    \mathcal P(\bbs)^d &\cong L_{q^2}(\tfrac12 -d)\oplus L_{q^2}(\tfrac32+d)\oplus E(\mathcal P(\bbs)^{d-1}),\\
    &\cong L_{q^2}(\tfrac12 -d)\oplus L_{q^2}(\tfrac32 + d)\oplus \bigoplus_{k=0}^{d-1} E^{k+1}(\mathcal M^{d-1-k}),
\end{align*}
where we have used the induction hypothesis in the last isomorphism. The result now follows from Proposition \ref{p:SpMonog}, together with the observation that in a completely reducible $\Cu_q$-module which is decomposed in terms of lowest weight irreducible modules, the intersection of  $\ker(F)$ with $\im(E)$ is trivial. \end{proof}
\begin{remark}
    Our usage of the subscripts $\Cm_{\pm}^d$ is because when $d=0$ they refer to even ($+$) and odd ($-$) polynomials in $\bbs=\bbc[y]$. It is unfortunate that the sign in the corresponding lowest weights are reversed.
\end{remark}

\noindent Theorem \ref{thm:TriangularDecomp} implies a complete analogue of the classical Fischer triangle \cite[Theorem 5.4]{DBSS14}.

\begin{corollary}\label{cor:FischerTriang}
    The joint action of $(\Cu_{q^2}, \Cu_q)$ decomposes $\mathcal P(\bbs)$ as 
    $$\mathcal P(\bbs)= \bigoplus_{d\geq 0} \left(L_{q^2}(\tfrac{1}{2}-d)\oplus L_{q^2}(\tfrac{3}{2}+d)\right) \otimes L_q(2d+2).$$
     This decomposition is expressed as the triangle
\[
\begin{tikzcd}
\Cp(\bbs)^0\deq     & \Cp(\bbs)^1\deq           & \Cp(\bbs)^2\deq             & \Cp(\bbs)^3\deq             & \cdots\\
\Cm^0 \arrow[r,"E"] & E(\Cm^0)\arrow[r,"E"]\dop & E^2(\Cm^0)\arrow[r,"E"]\dop & E^3(\Cm^0)\arrow[r,"E"]\dop & \cdots\\
                    & \Cm^1\arrow[r,"E"]        & E(\Cm^1)\arrow[r,"E"]\dop   & E^2(\Cm^1)\arrow[r,"E"]\dop & \cdots\\
                    &                           & \Cm^2\arrow[r,"E"]          & E(\Cm^2)\arrow[r,"E"]\dop   & \cdots\\
                    &                           &                             & \Cm^3\arrow[r,"E"]          & \cdots\\
\end{tikzcd}
\]
with $\Cm^d = \Cm^d_+\oplus \Cm^d_-$ as in Theorem \ref{thm:TriangularDecomp}.\hfil\qed
\end{corollary}

\section{Projection onto the monogenics}\label{s:projection}
For each $d\geq 0$, given the decomposition of $\Cp(\bbs)^d$ of \eqref{eq:HomogDecomp}, let $\Pi_d$ denote the projection $\Cp(\bbs)^d \to \Cm^d$ onto the space of symplectic monogenics in degree $d$. In this section, we will give a precise expression for $\Pi_d$. We will borrow ideas from the theory of extremal projectors (see, e.g., the survey paper \cite{T11} and the references therein), but, instead of describing these projections as an element in some suitable completion of the quantum group, we will proceed in an ad hoc manner and compute said projections by solving some finite system of equations. 

To that end, note first that since $F^{d+1}(\Cp(\bbs)^d) = 0$, we search for operators of the type
\begin{equation}\label{eq:Projectors}
\Pi_d = c_0\textup{Id} + \sum_{j=1}^d c_jE^jF^j
\end{equation}
for certain complex coefficients $\{c_j\mid 0\leq j \leq d\}$. We shall find these coefficients by solving the system of equations
\begin{equation}\label{eq:system}
    \Pi_d(\mathbf{m}_d) = \mathbf{m}_d,\qquad \Pi_d(E^j(\mathbf{m}_{d-j}))=0, \qquad j=1,2,\ldots,d,
\end{equation}
where $\mathbf{m}_{d-j}\in\Cm^{d-j}$ for each $j$. Note also that since $\sigma_\Delta(\Cu_{q^2})$ commutes with $E$ and $F$, the operators $\Pi_d$ will commute with the diagonal action $\sigma_\Delta$. We have the following result.

\begin{theorem}\label{t:Projection}
    For each $d\geq 0$, the $\Cu_{q^2}$-equivariant projection $\Pi_d:\Cp(\bbs)^d \to \Cm^d$ is given by
    \[
        \Pi_d = \sum_{j=0}^d\frac{[2d-j]_q!}{[j]_q![2d]_q!}E^jF^j.
    \]
\end{theorem}

\begin{proof}
We show, by induction, that the coefficients $c_j$ for $j=0,1,\ldots, d$ are given as in the statement. The claim is clearly true for $j=0$. Assuming for $j$, let $\mathbf{m}\in \Cm^{d-(j+1)}$. Then, the equation $\Pi_d(E^{j+1}(\mathbf{m}))=0 $ implies
\[
0 = E^{j+1}(\mathbf{m})+ \sum_{r=1}^j \frac{[2d-r]_q!}{[r]_q![2d]_q!} E^r F^r E^{j+1}(\mathbf{m}) + c_{j+1}E^{j+1}F^{j+1}E^{j+1}(\mathbf{m}).
\]
Using Lemma \ref{l:FEcommutations} and noting that  $\{q^{j+1-n}K\}_q(\textbf{p}) = [2d-j-n+1]_q\textbf{p}$ for any $n\in\bbz$ and $\textbf{p} \in \Cp(\bbs)^{d-j-1}$, we obtain
\[
c_{j+1} = \sum_{r=0}^j(-1)^{r+j}\frac{[2d-r]_q![2d-2j-1]_q!}{[r]_q![2d]_q![j+1-r]_q![2d-j-r]_q!}.
\]
To finish the proof of this proposition, we must show that
\[
\sum_{r=0}^j(-1)^{r+j}\frac{[2d-r]_q![2d-2j-1]_q!}{[r]_q![2d]_q![j+1-r]_q![2d-j-r]_q!} = \frac{[2d-j-1]_q!}{[j+1]_q![2d]_q!}.
\]
Multiplying both sides of the equation by $(-1)^j[2d]_q!$, note that this is equivalent to proving 
    \[
    0=\sum_{r=0}^{j+1}(-1)^r \frac{[2d-r]_q![2d-2j-1]_q!}{[r]_q![2d-r-j]_q![j+1-r]_q!}.
    \]
    Multiplying both sides by the scalar $\tfrac{[j+1]_q!}{[j]_q![2d-2j-1]_q!}$ we are left with proving the identity
    \[
    \sum_{r=0}^{j+1}(-1)^r \sqbinom{j+1}{r}_q\sqbinom{2d-r}{j}_q = 0.
    \]
    Changing variables from $r$ to $(j+1-\ell)$ and using the symmetry of the binomial coefficients in the lower argument, this is equivalent to the equation
    \[
    \sum_{\ell=0}^{j+1}(-1)^\ell \sqbinom{j+1}{\ell}_q\sqbinom{(2d-j-1)+\ell}{j}_q = 0.
    \]
The result now follows from Lemma \ref{l:BinomId} with $N=2d-j-1\geq d\geq j+1$ and $n=j+1$.
\end{proof}

\section{Generalized Symmetries}\label{s:GenSym}
If $D \in \End(V)$ is an operator on a vector space $V$, we say that $T\in\End(V)$ is a generalized symmetry of $D$ if there exists $S\in\End(V)$ such that
\[
DT = SD.
\]
In such a situation, observe that a generalized symmetry necessarily preserves the kernel of $D$. In this section, we will make use of the notation $\Ci_D$ to denote the left ideal generated by the operator $D$. Note that $T$ is a generalized symmetry of $D$ if and only if $DT\equiv 0 \mod \Ci_D$. 

Back to the setting of the metaplectic duality, in this section we shall study some generalized symmetries of $F$, other than those arising from the diagonal action of $\Cu_{q^2}$. Before diving into this task, first note that, similarly to the classical case, the adjoint action of $\Cu_{q^2}$ on $\Cw_q(2)\otimes \Cc_{s,q}$ naturally contains two copies of the standard representation. We summarize this information in the following proposition.

\begin{proposition}\label{p:adstd}
    The adjoint action of $\Cu_{q^2}$ on $\Cw_q(2)\otimes \Cc_{s,q}$ via the diagonal representation $\sigma_\Delta$ satisfy the following table of commutations.
    \begin{center}
    \begin{tabular}{R | C  C  C  C}
         & \mu_1 & \gamma_1^2\mu_2 & \partial_{2,q^2} & \gamma_2^{-2}\partial_{1,q^2}\\[1.5mm]\hline\\[-3mm]
        \ad_{K_{\Delta}}& q^2\mu_1 & q^{-2}(\gamma_1^2\mu_2) & q^2\partial_{2,q^2} & q^{-2}(\gamma_2^{-2}\partial_{1,q^2})\\[1.5mm]
        \ad_{E_{\Delta}}& 0 &\mu_1 & 0 & -q^2\partial_{2,q^2}\\[1.5mm]
        \ad_{F_{\Delta}}&\gamma_1^2\mu_2& 0 & -q^{-2}(\gamma_2^{-2}\partial_{1,q^2}) & 0
    \end{tabular} .
    \end{center}
    In particular, the spaces $\Cx = \textup{span}\{\mu_1,\gamma_1^2\mu_2\}$ and $\Cd = \textup{span}\{\partial_{2,q^2},\gamma_2^{-1}\partial_{1,q^2}\}$ are isomorphic to the standard representation $L_{q^2}(-1)$ of $\Cu_{q^2}$.
\end{proposition}

\begin{proof}
    The last part is immediate from the table of relations.
    Using \ref{eq:adjointact}, relations \ref{eq:qWeylRelGam}--\ref{eq:qWeylRelCom2} of Hayashi's $q$-Weyl algebra and Proposition \ref{p:qWeylComm} we have $\ad_{K_\Delta}(\mu_1) = q^2\mu_1$, $\ad_{E_\Delta}(\mu_1) = [E_\Delta,\mu_1]K_\Delta^{-1} = 0$ and $\ad_{F_\Delta}(\mu_1) = [F_\Delta,\mu_1]_{q^{-2}} = [\mu_2\partial_{1,q^2},\mu_1]_{q^{-2}} = \gamma_1^2\mu_2$. The other columns are proved in a similar fashion. 
\end{proof}

Aside from generating a copy of the standard representation of $\Cu_{q^2}$ inside $\Cw_{q}(2)\otimes \Cc_{s,q}$ via the adjoint action, the elements $\gamma_2^{-2}\partial_{1,q^2},\partial_{2,q^2}$ also commute with $F$. Hence, we have the following observation, whose proof is immediate.

\begin{proposition}
    The elements $\gamma_2^{-2}\partial_{1,q^2},\partial_{2,q^2}$ are generalized symmetries of  $F$ such that
    \[
    \gamma_2^{-2}\partial_{1,q^2},\partial_{2,q^2}:\Cm^d\to\Cm^{d-1}. 
    \] \hfil\qed
\end{proposition}

Next, we focus on operators that satisfy $\Cm^d\to\Cm^{d+1}$, for all $d\geq 0$. Certainly, $\mu_1$ and $\gamma_1^2\mu_2$ send $\Cm^d\to\Cp(\bbs)^{d+1}$, but their image does not land in $\Cm^{d+1}$, in general. We can amend this situation by composing with the projection $\Pi_{d+1}:\Cp(\bbs)^{d+1}\to \Cm^{d+1}$. We now work towards removing the dependence on the polynomial degree $d$ and globally define operators on $\Cp(\bbs)$ that preserve $\Cm$ by sending $\Cm^d\to\Cm^{d+1}$ for all $d\geq 0$. For that, we need a preliminary observation, which is akin to the classical situation (see \cite[Section 3]{DBHS17}). Recall the notation $\Ci_F$ to denote the left ideal generated by $F$ in $\Cw_q(2)\otimes \Cc_{s,q}$.

\begin{proposition}
    Computing modulo $\Ci_F$, the following hold:
    \begin{equation}\label{eq:adF-formulas}
        \begin{aligned}
            F\mu_1 & \equiv \gamma_1^2\nu \mod \Ci_F,\\
            F^2\mu_1 & \equiv q(q-q^{-1})\gamma_1^2\partial_{1,q^2}\nu^2 - \omega^2\partial_{2,q^2}\mod \Ci_F,\\
            F^3\mu_1 & \equiv 0\mod\Ci_F.
        \end{aligned}
    \end{equation}
\end{proposition}
\begin{proof}
    Since $\ad_F(X) = FX - (K^{-1}XK)F$, we have $FX \equiv \ad_F(X)\mod\Ci_F$, for any $X$ in $\Cw_q(2)\otimes \Cc_{s,q}$. The result follows by straightforwardly computing $\ad_{F^n}(\mu_1)$ for $n = 1,2,3$.
    \details{With more details, we compute
\[
\ad_F(\mu_1) = [F,\mu_1]_{q^{-2}}
=[\partial_{1,q^2}\nu,\mu_1]_{q^{-2}}
+[\gamma_1^{-2}\omega\partial_{2,q^2}\nabla,\mu_1]_{q^{-2}} = \gamma_1^2\nu.
\]
Furthermore, $\ad_{F^2}(\mu_1) = \ad_F(\ad_F(\mu_1))$, and hence we compute
\begin{align*}
    \ad_F(\gamma_1^2\nu) &= [\partial_{1,q^2}\nu,\gamma_1^2\nu]-[\gamma_1^{-2}\omega\partial_{2,q^2}\nabla,\gamma_1^2\nu]\\
    &=(q^2-1)\gamma_1^2\partial_{1,q^2}\nu^2 - \omega\partial_{2,q^2}[\nabla,\nu]_{q^{-1}}\\
    &=(q^2-1)\gamma_1^2\partial_{1,q^2}\nu^2 - \omega^2\partial_{2,q^2}.
\end{align*}
Finally, we compute $\ad_{F^3}(\mu_1) = \ad_F(\ad_F(\gamma_1^2\nu)) = A_1 - A_2$, with
\begin{align*}
    A_1 &= (q^2-1)\ad_F(\gamma_1^2\partial_{1,q^2}\nu^2)\\
    &=(q^2-1)\left([\partial_{1,q^2}\nu,\gamma_1^2\partial_{1,q^2}\nu^2]_{q^2}-[\gamma_1^{-2}\omega\partial_{2,q^2}\nabla,\gamma_1^2\partial_{1,q^2}\nu^2]_{q^2}\right)\\
    &=(q^2-1)(-\omega\partial_{1,q^2}\partial_{2,q^2}[\nabla,\nu^2]_{q^{-2}})\\
    &=-(q^{2}-q^{-2})\omega^2\partial_{1,q^2}\partial_{2,q^2}\nu
\end{align*}
(we used \eqref{eq:CommId0a} in the last line) and
\begin{align*}
    A_2 &= \ad_F(\omega^2\partial_{2,q^2})\\
    &=[\partial_{1,q^2}\nu,\omega^2\partial_{2,q^2}]_{q^2}-[\gamma_1^{-2}\omega\partial_{2,q^2}\nabla,\omega^2\partial_{2,q^2}]_{q^2}\\
    &=-(q^{2}-q^{-2})\omega^2\partial_{1,q^2}\partial_{2,q^2}\nu.
\end{align*}
This finishes the proof.}
\end{proof}
The conclusion thus far is that, for any $\mathbf{m}\in\Cm^{d}$ we can write
\[
\Pi_{d+1}(\mu_1(\mathbf{m}))
    = \mu_1(\mathbf{m}) + \frac{\left(E\ad_F(\mu_1)\right)(\mathbf{m})}{[2d+2]_q}+ \frac{\left(E^2\ad_{F^2}(\mu_1)\right)(\mathbf{m})}{[2]_q[2d+2]_q[2d+1]_q}.
\]
Note, however, that $\{q^nK\}_q(\mathbf{m}) = [2d+2 + n]_q \mathbf{m}$, for any $n\in \bbz$. Hence, working in some completion of $\Cw_q(2)\otimes \Cc_{s,q}$ where we allow divisions (on the right side) by Laurent polynomials in $K,K^{-1}$, it follows that we can write
\[
\Pi_{d+1}\circ\mu_1(\mathbf{m})
    = \left(\mu_1 + \frac{E\ad_F(\mu_1)}{\{K\}_q}+ \frac{E^2\ad_{F^2}(\mu_1)}{[2]_q\{K\}_q\{q^{-1}K\}_q}\right)(\mathbf{m}).
\]
Clearing out the denominator so to obtain an element in $\Cw_q(2)\otimes \Cc_{s,q}$, we define
\begin{equation}\label{eq:Z1}
    Z_1 := [2]_q\mu_1\{K\}_q\{q^{-1}K\}_q + [2]_q E\ad_F(\mu_1)\{q^{-1}K\}_q + E^2\ad_{F^2}(\mu_1).
\end{equation}
Now, as $[F_\Delta,X]=0$ for any $X\in\sigma(\Cu_q)$, it follows that $\ad_{F_\Delta}\ad_X = \ad_X\ad_{F_\Delta}$. Further, recall from Proposition \ref{p:adstd} that $\ad_{F_\Delta}(\mu_1) = \gamma_1^2\mu_2$ and that $\ad_{F^3}(\gamma_1^2\mu_2)=\ad_{F_\Delta}(\ad_{F^3}(\mu_1))=0$. So, let us also consider
\begin{equation}\label{eq:Z2}
    Z_2 := \ad_{F_\Delta}(Z_1) = [2]_q(\gamma_1^2\mu_2)\{K\}_q\{q^{-1}K\}_q + [2]_q E\ad_F(\gamma_1^2\mu_2)\{q^{-1}K\}_q + E^2\ad_{F^2}(\gamma_1^2\mu_2).
\end{equation}

\begin{theorem}
    The elements $Z_1,Z_2$ of \eqref{eq:Z1}, \eqref{eq:Z2} are generalized symmetries of $F$ such that 
    \[
    Z_1,Z_2:\Cm^d\to\Cm^{d+1}
    \] 
    for all $d\geq 0$. Furthermore, via $\ad^{\sigma_\Delta}$, their linear span $\Cz = \textup{span}\{Z_1,Z_2\}$ is isomorphic to the standard representation $L_{q^2}(-1)$ of $\Cu_{q^2}$.
\end{theorem}
\begin{proof}
    By construction, we have $Z_1,Z_2:\Cm^d\to \Cm^{d+1}$ and, using Proposition \ref{p:adstd} and the fact that $\sigma_\Delta$ commutes with $\sigma$, we have $\Cz\cong_{\Cu_{q^2}}L_{q^2}(-1)$. It remains to prove that $Z_1,Z_2$ are generalized symmetries of $F$. To that end, we will show that, for $j\in\{1,2\}$, we have $\ad_F(Z_j) \equiv 0 \mod \Ci_F$, since this implies that $FZ_j = (K^{-1}Z_jK)F + \ad_F(Z_j)\equiv 0 \mod\Ci_F$. Using that $\sigma_\Delta$ and $\sigma$ commute, observe that 
    \[
    \ad_F(Z_2) = \ad_F\ad_{F_\Delta}(Z_1) = \ad_{F_\Delta}\ad_F(Z_1) = F_\Delta\ad_F(Z_1) - K_\Delta^{-1}\ad_F(Z_1)K_\Delta F_\Delta. 
    \]
    So it suffices to show that $\ad_F(Z_1)\equiv 0\mod \Ci_F$. But, in order to do so, we compute directly
    \begin{align*}
        \ad_F(E) &= -\{K\}_q + (1-q^{-2})EF,\\
        \ad_F(\{q^nK\}_q) &\equiv 0\mod\Ci_F,&(\textup{for all } n\in\bbz)\\
    \ad_F(E^2) &\equiv -[2]_qE\{qK\}_q \mod \Ci_F,\\
        \ad_F(E)\ad_F(\mu_1) &\equiv -\{K\}_q\ad_F(\mu_1) + (1-q^{-2})E\ad_{F^2}(\mu_1)\mod \Ci_F,
    \end{align*}
    so that using the twisted derivation property $\ad_F(XY) = \ad_F(X)Y + (K^{-1}XK)\ad_F(Y)$ for all $X,Y\in\Cw_q(2)\otimes\Cc_{s,q},$ we get
    \begin{align}\label{eq:gSymEq1}
        \ad_F(\mu_1\{K\}_q\{q^{-1}K\}_q)  &\equiv (\ad_F\mu_1)\{K\}_q\{q^{-1}K\}_q \mod\Ci_F\\\label{eq:gSymEq2}
        \ad_F(E\ad_F(\mu_1)\{q^{-1}K\}_q) &\equiv \left(-\ad_F(\mu_1)\{K\}_q + E\ad_{F^2}(\mu_1)\right)\{q^{-1}K\}_q \mod\Ci_F\\\nonumber
        \ad_F(E^2\ad_{F^2}(\mu_1)) &\equiv -[2]_qE\{qK\}_q\ad_{F^2}(\mu_1) \mod \Ci_F\\\label{eq:gSymEq3}
        &= -[2]_qE\ad_{F^2}(\mu_1)\{q^{-1}K\}_q,
    \end{align}
    where we used in the second line that $K\ad_{F}(\mu_1) = \ad_{F}(\mu_1)K$  and in the last line that $K\ad_{F^2}(\mu_1) = q^{-2}\ad_{F^2}(\mu_1)K$ (see the formula for $\ad_{F^n}(\mu_1)$ in \eqref{eq:adF-formulas}). Using \eqref{eq:gSymEq1}--\eqref{eq:gSymEq3}, it follows immediately that $\ad_F(Z_1)\equiv 0$ modulo $\Ci_F$, and we are done.
\end{proof}

\section*{Acknowledgements}
We would like to thank H. De Bie for indicating to us the reference \cite{DBHS17} and for encouraging us to work on this topic.
MB and MDM were partially supported by a CNPq grant 405793/2023-5 and MDM was supported by the special research fund (BOF) from Ghent University [BOF20/PDO/058].
Substantial parts of this paper were developed while MB visited Ghent University and when MDM visited the Federal University of Paran\'a. They would like to kindly thank all the staff in both institutions for the support leading to excellent working conditions. They also acknowledge the generous support of a CWO grant from Ghent University to foster international collaborations.

\end{document}